\documentclass[10pt, a4paper]{article}
\usepackage{graphicx,latexsym, amsmath,amsfonts}
\usepackage{amsthm}

\usepackage{verbatim}

\usepackage{setspace}

\makeatletter \@addtoreset{equation}{section}

\begin{document}

\linespread{1.3}

\newcommand{\E}{\mathbb{E}}
\newcommand{\PP}{\mathbb{P}}
\newcommand{\RR}{\mathbb{R}}
\newcommand{\NN}{\mathbb{N}}

\newtheorem{theorem}{Theorem}[section]
\newtheorem{lemma}[theorem]{Lemma}
\newtheorem{coro}[theorem]{Corollary}
\newtheorem{defn}[theorem]{Definition}
\newtheorem{assp}[theorem]{Assumption}
\newtheorem{expl}[theorem]{Example}
\newtheorem{prop}[theorem]{Proposition}
\newtheorem{rmk}[theorem]{Remark}
\newtheorem{notation}[theorem]{Notation}

\newcommand\tq{{\scriptstyle{3\over 4 }\scriptstyle}}
\newcommand\qua{{\scriptstyle{1\over 4 }\scriptstyle}}
\newcommand\hf{{\textstyle{1\over 2 }\displaystyle}}
\newcommand\hhf{{\scriptstyle{1\over 2 }\scriptstyle}}

\newcommand{\eproof}{\indent\vrule height6pt width4pt depth1pt\hfil\par\medbreak}

\def\a{\alpha} \def\g{\gamma}
\def\e{\varepsilon} \def\z{\zeta} \def\y{\eta} \def\o{\theta}
\def\vo{\vartheta} \def\k{\kappa} \def\l{\lambda} \def\m{\mu} \def\n{\nu}
\def\x{\xi}  \def\r{\rho} \def\s{\sigma}
\def\p{\phi} \def\f{\varphi}   \def\w{\omega}
\def\q{\surd} \def\i{\bot} \def\h{\forall} \def\j{\emptyset}

\def\be{\beta} \def\de{\delta} \def\up{\upsilon} \def\eq{\equiv}
\def\ve{\vee} \def\we{\wedge}

\def\t{\tau}

\def\F{{\cal F}}
\def\T{\tau} \def\G{\Gamma}  \def\D{\Delta} \def\O{\Theta} \def\L{\Lambda}
\def\X{\Xi} \def\S{\Sigma} \def\W{\Omega}
\def\M{\partial} \def\N{\nabla} \def\Ex{\exists} \def\K{\times}
\def\V{\bigvee} \def\U{\bigwedge}

\def\1{\oslash} \def\2{\oplus} \def\3{\otimes} \def\4{\ominus}
\def\5{\circ} \def\6{\odot} \def\7{\backslash} \def\8{\infty}
\def\9{\bigcap} \def\0{\bigcup} \def\+{\pm} \def\-{\mp}
\def\<{\langle} \def\>{\rangle}

\def\lev{\|} \def\rev{\|}
\def\1{\mathbf{1}}

\def\tl{\tilde}
\def\trace{\hbox{\rm trace}}
\def\diag{\hbox{\rm diag}}
\def\for{\quad\hbox{for }}
\def\refer{\hangindent=0.3in\hangafter=1}

\newcommand\wD{\widehat{\D}}

\pagestyle{plain}

\title
{ \bf Strong convergence and stability of implicit numerical methods for 
 stochastic differential equations with non-globally Lipschitz continuous coefficients.}

\author{ Xuerong Mao%
        \thanks{%
                  Department of Mathematics and Statistics,
                  University of Strathclyde,
                   Glasgow, G1 1XH, Scotland, UK
                  (\texttt{x.mao@strath.ac.uk}).
                 }
\and
Lukasz Szpruch%
          \thanks{%
                  Mathematical Institute,
                  University of Oxford,
                   Oxford, OX1 3LB, UK
                  (\texttt{lukas.szpruch@maths.ox.ac.uk}).
                  }
}

\date{}

\maketitle

\thispagestyle{empty}


\begin{abstract}
\textsf{\em We are interested in the strong convergence and almost sure stability of 
Euler-Maruyama (EM) type approximations to the solutions of
stochastic differential equations (SDEs) with non-linear and non-Lipschitzian
coefficients. Motivation comes from finance and biology where many widely applied models do not satisfy the standard
assumptions required for the strong convergence. In addition we examine the globally almost surely asymptotic
stability in this non-linear setting for EM type schemes. In particular, we present a stochastic counterpart of the discrete
LaSalle principle from which we deduce stability properties for numerical methods.}



\medskip
\noindent{\small\bf AMS Subject Clasification: } 65C30,\;65L20,\;60H10

\end{abstract}

\section{Introduction}

Throughout this paper, let $(\Omega, {\mathcal{F}},\{{\mathcal{F}}_t\}_{t\geq 0}, \PP)$ be a
complete probability space with a filtration
$\{{\mathcal{F}}_t\}_{t\geq 0}$ satisfying the usual conditions (that is to say,
it is right continuous and increasing while
${\mathcal{F}}_0$ contains all $\PP$-null sets).  Let $w(t)=(w_{1}(t),...,w_{d}(t))^{T}$
be a $d$-dimensional Brownian motion defined on the probability space, where $T$ denotes
the transpose of a vector or a matrix.
In this paper we study the numerical approximation of the stochastic differential equations (SDEs)
\begin{equation}   \label{eq:SDE}
dx(t)=f(x(t))dt+g(x(t))dw(t).
\end{equation}
Here $x(t)\in \RR^{n}$ for each $t\ge 0$ and $f:\RR^{n}\rightarrow\RR^{n}$
and $g:\RR^{n}\rightarrow \RR^{n\times d}$. For
simplicity we assume that $x_{0}\in \RR^{n}$ is a deterministic vector.
Although the method of Lyapunov functions allows us to show that there are solutions to
 a very wide family of SDEs (see e.g. \cite{chas1980stochastic,mao2007stochastic}), in general,
both the explicit solutions and the probability distributions of
the solutions are not known.  We therefore consider computable discrete
approximations that, for example, could be used in Monte Carlo simulations.
Convergence and stability
of these methods are well understood for SDEs with Lipschitz continuous
coefficients: see \cite{kloeden1992numerical} for example.
Our primary objective is to study classical strong convergence and stability questions for numerical
approximations in the case where $f$ and $g$
are not globally Lipschitz continuous.
A good motivation for our work is an instructive conditional result of Higham et al. \cite{higham2003strong}.
Under the local Lipschitz condition, they proved that
uniform boundedness of moments of both the
solution to \eqref{eq:SDE} and its approximation are sufficient for strong
convergence. That immediately raises the question of what type of conditions
can guarantee such a uniform boundedness of moments. It is well known that
the classical linear growth condition is sufficient to bound the
moments for both the true solutions and their EM approximation \cite{kloeden1992numerical, mao2007stochastic}. It is
also known that when we try to bound the moment of the true solutions,  a useful way
to relax the linear growth condition is to apply the Lyapunov-function technique,
with $V(x)=\lev x\rev^{2}$. This leads us to the monotone condition \cite{mao2007stochastic}.
More precisely, if there exist constants $\a,\be > 0$ such that
the coefficients of equation \eqref{eq:SDE} satisfy
\begin{equation}
\<x,f(x)\>+\frac{1}{2}\lev g(x)\rev^{2}\le \a + \be\lev x\rev^2 \quad \mathrm{~for~all~~}x \in\RR^{n},
    \label{int:mono}
\end{equation}
then
\begin{equation} \label{eq:bex}
\sup_{0\leq t\leq T} \E\lev x(t)\rev ^{2}< \8 \quad \forall T > 0.
\end{equation}
Here, and throughout, $\lev x \rev $ denotes both the Euclidean vector norm and the Frobenius
matrix norm and $\<x,y\>$ denotes the scalar product of vectors $x,y\in\RR^{n}$.
However, to the best of our knowledge there is no result on the moment bound for the numerical solutions of SDEs
under the monotone condition (\ref{int:mono}).
Additionally, Hutzenthaler et al. \cite{hutzenthaler2011strong} proved that in the
case of super-linearly growing coefficients,
the EM approximation
may not converge in the strong $L^{p}$-sense nor in the
 weak sense
to the exact solution.
For example, let us consider   a non-linear SDE
\begin{equation}\label{eq:ex}
 dx(t) = (\mu -\a x(t)^3 ) dt + \be x(t)^2 dw(t),
\end{equation}
where $\mu,\a,\be\ge0$ and $\a>\frac{1}{2}\be^{2}$.
In order to approximate SDE \eqref{eq:ex} numerically, for any $\D t$, we define
the partition  $\mathcal{P}_{\D t}:=\{t_{k}=k\Delta t :k=0,1,2,\ldots,N\}$ of the time interval
$[0,T]$, where $N\D t=T$ and $T>0$. Then we define the EM approximation $Y_{t_{k}} \approx x(t_{k})$
of \eqref{eq:ex} by
\begin{equation} \label{eq:int:EMex}
Y_{t_{k+1}} = Y_{t_{k}} +( \mu -\a Y_{t_{k}}^{3}) \D t
                 +\be Y_{t_{k}}^{2}\D w_{t_{k}},
\end{equation}
where $\D w_{t_{k}} = w(t_{k+1})-w(t_{k})$.
It was shown in \cite{hutzenthaler2011strong} that
\[
\lim_{\D t \rightarrow 0}\E \lev Y_{t_{N}}\rev^{2} = \8.
\]
On the other hand, the coefficients of \eqref{eq:ex} satisfy the monotone condition \eqref{int:mono}
so \eqref{eq:bex} holds.
Hence, Hutzenthaler et al. \cite{hutzenthaler2011strong} concluded that
\begin{align*}
 \lim_{\D t \rightarrow 0}\E \lev x(T)- Y_{t_{N}} \rev^{2}=\8.
\end{align*}
It is now clear that to prove the strong convergence theorem under condition \eqref{int:mono} it is
necessary to modify
the EM scheme.
Motivated by the existing works \cite{higham2003strong} and \cite{hu1996semi} we consider
implicit schemes. These authors have
demonstrated that a backward Euler-Maruyama (BEM) method strongly converges to
the solution of the SDE with one-sided Lipschitz drift and linearly growing diffusion
coefficients.  So far, to the best of our knowledge, most of the existing results on
the strong convergence for numerical schemes only 
cover SDEs where the diffusion coefficients have at
most linear growth \cite{berkaoui2007euler,mao2007approximations,higham2005convergence,hutzenthaler2010strong,kloeden1992numerical}.
However, the problem remains essentially unsolved for the important class of SDEs with
super-linearly growing diffusion coefficients.
We are interested in relaxing  the conditions for the diffusion coefficients to
justify Monte Carlo simulations for highly non-linear systems that arise in financial mathematics,
\cite{ahn1999parametric,campbell1998econometrics,ait1996testing,chan1992empirical,heston1997simple,lewis2000option},
 for example
\begin{equation}
dx(t)=(\mu -\a x^{r}(t))dt+\be x^{\rho}(t)dw(t), \quad r,\rho>1,
\end{equation}
where $\mu,\a,\be>0$, or in stochastic population dynamics
 \cite{mao2002environmental,bahar2004stochastic,mao2003asymptotic,pang-asymptotic,gard1988introduction},
 for example
 \begin{equation} \label{eq:Lotka-Voltera}
dx(t)=\diag(x_{1}(t),x_{2}(t),...,x_{n}(t))[(b+A x^{2}(t))dt+x(t)dw(t)],
\end{equation}
where $b=(b_{1},\ldots,b_{n})^{T}$, $x^{2}(t)= (x^{2}_{1}(t),\ldots,x^{2}(t)_{n})^{T}$ and matrix $A=[a_{ij}]_{1\le i,j \le n}$
is such that $\l_{max} (A + A^{T}) <0 $, where $\l_{max}(A)= \sup_{x\in\RR^{n}, \lev x \rev=1}x^{T}Ax$.
The only results we know, where the strong convergence of the numerical approximations
was considered for super-linear diffusion, is Szpruch et al. \cite{szpruchnumerical} and Mao and Szpruch \cite{szpruch-diss}.
In \cite{szpruchnumerical} authors have considered the BEM approximation for the following
scalar SDE which arises in finance \cite{ait1996testing},
\begin{equation*}
dx(t)=(\alpha _{-1}x(t)^{-1}-\alpha _{0}+\alpha
_{1}x(t)-\alpha_{2}x(t)^{r})dt+ \sigma x(t)^{\rho }dw(t) \quad r,\rho>1.
\end{equation*}
In \cite{szpruch-diss}, this analysis was extended to the multidimensional case under specific conditions
for the drift and  diffusion coefficients.  
In the present paper, we aim to prove strong convergence under general monotone condition \eqref{int:mono} in a multidimensional setting.
We believe that this condition is optimal for boundedness of moments of the implicit schemes.  
 The reasons that we are interested in the strong convergence are: a) the efficient variance reduction
 techniques, for example, the
multi-level Monte Carlo simulations  \cite{giles2008multilevel},  rely on the strong convergence properties;
 b) both weak convergence \cite{kloeden1992numerical}
and pathwise convergence \cite{kloeden2007pathwise} follow automatically.

Having established the strong convergence result we will proceed to the stability analysis of
the underlying numerical scheme for the non-linear SDEs (\ref{eq:SDE}) under the monotone-type condition.
The main problem concerns the propagation of an error during the simulation of an approximate path. If the numerical
scheme is not stable, then the simulated path may diverge substantially from the exact solution in
practical simulations.
Similarly, the expectation of the functional estimated by a Monte Carlo simulation may be significantly different
from that
of the expected functional of the underlying SDEs due to numerical instability.
Our aim here is to investigate almost surely asymptotic properties of the numerical schemes
for SDEs (\ref{eq:SDE}) via a stochastic
 version of the LaSalle principle.  LaSalle, \cite{lasalle1968stability}, improved
significantly the Lyapunov stability method for ordinary differential equations. Namely, he developed methods for
locating limit sets of nonautonomous systems \cite{hale1993introduction,lasalle1968stability}. The first stochastic
counterpart of his great achievement was
established by Mao \cite{mao1999stochastic} under the local Lipschitz and linear growth conditions.
Recently, this stochastic version was generalized by
Shen et al. \cite{shen2006improved}
to cover stochastic functional differential equations with locally Lipschitz continuous coefficients.
Furthermore, it is well known  that there exists a counterpart of the invariance principle for
discrete dynamical systems \cite{lasalle1976stability}. However,
there is no discrete stochastic counterpart of Mao's version of the LaSalle theorem.
In this work we investigate a special
case, with the Lyapunov function $V(x)=\lev x \rev^2$, of the LaSalle theorem. We shall show that the 
almost sure global stability can be easily
deduced from our results.
The primary objectives in our stability analysis are:  \vspace{-4mm}
\begin{itemize} 
  \item Ability to cover highly nonlinear SDEs; \vspace{-2mm}
  \item Mild assumption on the time step - $A(\a)$-stability concept \cite{higham2000stability}.  \vspace{-2mm}
\end{itemize} 
Results which investigate stability analysis for numerical methods can be found
in Higham \cite{higham2001mean,higham2000stability} for  the scalar
linear case, Baker et al. \cite{baker2005exponential} for the global Lipschitz and
 Higham et al. \cite{higham2003exponential} for one-sided
Lipschitz  and the linear growth condition. 
\newline
At this point, it is worth mentioning how our
work compares with that of Higham et al. \cite{higham2003strong}.
Theorem 3.3 in their paper is a very important contribution to the numerical SDE theory.
The authors proved strong convergence results for one-sided Lipschitz
and the linear growth condition on drift and diffusion coefficients, respectively.
What differentiates our work from \cite{higham2003strong} are:
a) We significantly relax the linear growth constraint on the diffusion coefficient and
we only ask for very general monotone type growth;
b) Our analysis is based on a more widely applied BEM scheme
 in contrast to the split-step scheme introduced in their paper.
An interesting alternative to the implicit schemes for numerical approximations of SDEs 
with non-globally Lipschitz drift coefficient recently appeared in \cite{hutzenthaler2010strong}. 
However the stability properties of this method are not analysed.  

In what follows, for notational simplicity, we use the convention that $C$ represent
 a generic positive constant independent of $\D t$, the value of which may be different for different appearances.

The rest of the paper is arranged as follows.
In section 2, we introduce the monotone condition under which we prove the existence of a unique solution to
equation (\ref{eq:SDE}), along with appropriate bounds that will be needed in further analysis.
In Section 3 we propose the $\o$-EM scheme, which  is known
as the BEM when $\o=1$, to approximate the solution of
equation (\ref{eq:SDE}). We show that
the $2nd$ moment of the $\o$-EM, can be bounded under
the monotone condition plus some mild assumptions on $f$ and $g$. In Section 4 we
introduce a new numerical method, which we call
the forward-backward Euler-Maruyama (FBEM). The FBEM scheme enables us to
overcome some measurability difficulties and avoid using Malliavin calculus. We
demonstrate that both the FBEM and the $\o$-EM do not differ much in the $L^{p}$-sense.
Then we prove a strong convergence theorem on a compact domain that is later
extended to the whole domain. We also perform
a numerical experiment that confirms our theoretical
results.
Section 5 contains the stability
analysis, where we prove a special
case of the stochastic
LaSalle theorem for discrete time processes.
\section{Existence and Uniqueness of Solution}
We require the coefficients $f$ and $g$ in \eqref{eq:SDE} to be locally
 Lipschitz continuous and to satisfy the monotone condition, that is
\begin{assp} \label{a0}
Both coefficients $f$ and $g$ in \eqref{eq:SDE} satisfy the following conditions:\\
\textit{\underline{Local Lipschitz condition}.}
For each integer $m\geq1$, there is a positive constant $C(m)$ such
that
\begin{equation*}
\lev f(x)-f(y) \rev +\lev g(x)-g(y) \rev \leq C(m)\lev x-y \rev
\end{equation*}
for those $x, y \in \RR^{n}$ with $\lev x \rev\vee \lev y\rev \leq m.$\\
\textit{\underline{Monotone condition}.} There exist constants $\a$ and $\be$ such that
\begin{equation} \label{as1}
\<x,f(x)\> +\frac{1}{2}\lev g(x)\rev^{2} \le \a+\be\lev x\rev^{2}
\end{equation}
for all $x\in\RR^{n}$.
\end{assp}
It is a classical result that under Assumption \ref{a0}, there
exists a unique solution to \eqref{eq:SDE} for any given initial value
$x(0)=x_{0}\in\RR^{n}$,  \cite{friedman1976,mao2007stochastic}. The reason why we
 present the following theorem with a proof here  is that
it reveals the upper bound for the probability that the process $x(t)$ stays on a
compact domain for finite time $T>0$. The bound will
be used to derive the main convergence theorem of this paper.
\begin{theorem} \label{Existence}
Let Assumption \ref{a0} hold. Then for any
given initial value $x(0)=x_{0}\in\RR^{n}$, there exists a unique, global
solution $\{x(t)\}_{t\ge0}$ to equation (\ref{eq:SDE}). Moreover, the solution has the properties that for any $T>0$,
\begin{equation}  \label{2.9}
\E\lev x(T)\rev ^2< ( \lev x_{0} \rev^{2}+2\a T )\exp(2\be T),
\end{equation}
and
\begin{equation} \label{eq:lim}
\PP(\tau _{m}\leq T)
\leq
\frac{( \lev x_{0} \rev^{2}+2\a T )\exp(2\be T)}{m^2},
\end{equation}
where $m$ is any positive integer and
\begin{equation} \label{eq:stop}
\tau _{m}=\inf \{t\ge 0 :\quad \lev x(t)\rev>m \}.
\end{equation}
\end{theorem}
\begin{proof} It is well known that under Assumption \ref{a0},  for any given initial
 value $x_0\in\RR^{n}$ there exists a unique
solution $x(t)$ to  the SDE (\ref{eq:SDE}), \cite{friedman1975stochastic,mao2007stochastic}.
Therefore we only need to prove that (\ref{2.9}) and (\ref{eq:lim}) hold.
Applying the It\^{o}
formula to the function $V(x,t)=\lev x\rev^{2}$,  we compute the diffusion
operator
\begin{align*} 
LV(x,t)
    &=  2\Big(\<x,f(x)\> +\frac{1}{2}
           \lev g(x)\rev^{2} \Big).
           \end{align*}
By Assumption \ref{a0}
\begin{equation} \label{const}
LV(x,t)\leq 2\a+2\be\lev x \rev^{2}.
\end{equation}
Therefore
\begin{equation*}
\E  \lev x(t\wedge \tau _{m})\rev ^{2}
\leq
  \lev x_{0} \rev^{2}+2\a T+\int_{0}^{t}2\be\E  \lev x(s\wedge \tau _{m}) \rev^{2}ds.
\end{equation*}
The Gronwall inequality gives
\begin{equation} \label{eq:bb}
\E  \lev x(T\wedge \tau _{m})\rev ^{2}
\leq
  ( \lev x_{0} \rev^{2}+2\a T )\exp(2\be T).
\end{equation}
Hence
\begin{equation*}
\PP(\tau _{m}\leq T)m^2
\leq
[ \lev x_{0} \rev^{2}+2\a T ]\exp(2\be T).
\end{equation*}
Next, letting $m\rightarrow \infty$  in \eqref{eq:bb}
and applying Fatou's lemma, we obtain
\begin{equation*}
\E\lev x(T)\rev ^{2} \leq  [ \lev x_{0} \rev^{2}+2\a T ]\exp(2\be T),
\end{equation*}
which gives the  other assertion (\ref{2.9}) and completes the proof.
\end{proof}
\section{The $\o$-Euler-Maruyama Scheme}
As indicated in the introduction, in order to approximate the solution of \eqref{eq:SDE} we will use the
$\o$-EM scheme.
Given any step size $\Delta
t$,  we define a partition $\mathcal{P}_{\Delta
t}:=\{t_{k}=k\Delta t :k=0,1,2,...\}$ of the half line
$[0,\8)$, and define
\begin{equation}\label{oEM}
X_{t_{k+1}}=X_{t_{k}}+\o f(X_{t_{k+1}}) \D t + (1-\o)f(X_{t_{k}}) \D t  + g(X_{t_{k}})\Delta w_{t_{k}},
\end{equation}
where $\Delta w_{t_{k}}=w_{t_{k+1}}-w_{t_{k}}$ and
$X_{t_{0}}=x_{0}$.
The additional parameter $\o \in [0,1]$ allows us to control the implicitness of the 
numerical scheme, that may lead to various
asymptotic behaviours of equation \eqref{oEM}.
For technical reasons we always require $\o \ge 0.5$.

Since we are dealing with an implicit scheme we need to make sure that  equation \eqref{oEM} has a
unique  solution $X_{t_{k+1}}$ given $X_{t_{k}}$.
To prove this, in addition to Assumption \ref{a0}, we
ask that function $f$ satisfies the one-sided Lipschitz condition.
\begin{assp} \label{os_lip}
\textit{One-sided Lipschitz condition.} There exists a constant $L>0$ such that
\begin{equation*}
\<x-y , f(x)-f(y) \>\le L \lev x-y\rev^{2} \quad \forall x,y \in \RR^{n}.
\end{equation*}
\end{assp}
It follows from the fixed point theorem that a unique solution $X_{t_{k+1}}$ to
equation (\ref{oEM}) exists given $X_{t_{k}}$, provided $\D t < \frac{1}{\o L}$, (see \cite{szpruch-diss} for more 
details). From now on we
always assume that $\D t < \frac{1}{\o L}$.
In order to implement numerical scheme \eqref{oEM} we define a function $F:\RR^{n}\rightarrow \RR^{n}$ as
\begin{equation} \label{eq:F}
 F(x) = x  - \o f(x)\D t.
\end{equation}
Due to Assumption \ref{os_lip}, there exists an inverse function
$F^{-1}$ and the solution to \eqref{oEM} can be represented in the following form
\[
 X_{t_{k+1}} = F^{-1} ( X_{t_{k}} + (1-\o)f(X_{t_{k}}) \D t  + g(X_{t_{k}})\Delta w_{t_{k}} ).
\]
Clearly, $X_{t_{k}}$ is $\F_{t_{k}}$-measurable.  In many applications,
the drift coefficient of the SDEs has a cubic or quadratic form, whence the inverse function
can be found explicitly. For more complicated SDEs we can find the inverse function $F^{-1}$ using
root-finding algorithms, such as Newton's method.
\subsection{Moment Properties of $\o$-EM}
In this section we show that the second moment of the $\o$-EM \eqref{oEM} is bounded (Theorem \ref{TL2}).
To achieve the bound we employ the stopping
time technique, in a similar way as in the proof of Theorem \ref{Existence}.
However, in discrete time approximations for a stochastic process, the problem of
overshooting the level where we would like to stop the
process  appears, \cite{buchmann2005simulation,broadie1997continuity,mannella1999absorbing}.

Due to the implicitness of  scheme \eqref{oEM}, an additional but mild restriction
on the time step appears.  That is,  from now on,  we require $\D t \le \D t^{*}$, where
$\D t^* \in (0, (max\{L,2\be\}\o)^{-1})$ with $\be$ and $L$  defined in Assumptions \ref{a0} and \ref{os_lip}, respectively.

The following lemma shows that in order to guarantee the boundedness of moments for
$X_{t_{k}}$ defined by \eqref{oEM} it is enough to bound the moments of
$F(X_{t_{k}})$, where $F$ is defined by \eqref{eq:F}.
\begin{lemma} \label{lem:F}
 Let Assumption \ref{a0} hold. Then for $F(x) = x  - \o f(x)\D t$ we have
\[
 \lev x \rev^{2} \le ( 1- 2 \be \o \D t)^{-1} \left[ \lev F(x) \rev^{2} + 2\o\a\D t\right] \quad \forall x \in \RR^{n}.
\]
\end{lemma}
\begin{proof}
 Writing $\lev F(x) \rev^{2} = \< F(x),F(x) \>  $ and using Assumption \ref{a0} we arrive at
\begin{align*}
 \lev F(x) \rev^{2} = & \lev x \rev^{2} - 2\o \<x,f(x)\> \D t + \o^{2}\lev f(x)\rev^{2}\D t^{2} \\
                  \ge & ( 1- 2 \be \o \D t) \lev x \rev^{2} - 2\o\a\D t, 
\end{align*}
and the assertion follows.
\end{proof}
We define the stopping time $\l_{m}$ by
\begin{equation} \label{sd}
\lambda _{m}=\inf \{k: \lev X_{t_{k}}\rev > m \}.
\end{equation}
We observe that when $k \in[0,\l_{m}(\w)]$,
$\lev X_{t_{k-1}}(\w) \rev \le m$,
but we may have  $\lev X_{t_{k}} (\w)   \rev > m$,
so the following lemma is not trivial.
\begin{lemma} \label{stopping}
Let Assumptions \ref{a0}, \ref{os_lip} hold, and $\o\ge 0,5$. Then for $p\geq 2$ 
and sufficiently large integer $m$, there exists a
constant $C(p,m)$, such that
\begin{equation*}
\E\left[\lev X_{t_{k}}\rev ^{p}\1_{[0,\lambda
_{m}]}(k)\right]< C(p,m) \quad \mathrm{~~for~any~~}k\ge0.
\end{equation*}
\end{lemma}
\begin{proof}
The proof is given in the Appendix.
\end{proof}
For completeness of the exposition we recall the discrete Gronwall inequality, that
 we will use in the proof of Theorem \ref{TL2}.
\begin{lemma}[The Discrete Gronwall Inequality]
Let M be a positive integer. Let $u_{k}$ and $v_{k}$ be non-negative numbers for k=0,1,...,M. If
\[
u_{k}\le u_{0} + \sum_{j=0}^{k-1}v_{j}u_{j}, \qquad \forall k=1,2,...,M,
\]
then
\[
u_{k}\le u_{0} \exp\left( \sum_{j=0}^{k-1} v_{j}\right), \qquad \forall k=1,2,...,M.
\]
\end{lemma}
The proof can be found in Mao et al. \cite{mao2006stochastic}.
To prove the boundedness of the second moment for the $\o$-EM \eqref{oEM}, we need an additional but mild
assumption on the coefficients $f$ and $g$.
\begin{assp}\label{as:polynomial}
The coefficients of equation (\ref{eq:SDE})
satisfy the polynomial growth condition. That is,  there exists a pair of constants  $h\ge 1$ and $C(h)>0$  such that
\begin{equation} \label{ass:P}
\lev f(x) \rev \vee \lev g(x)\rev \leq C(h) ( 1 + \lev x \rev ^{h} ), \quad \forall x\in\RR^{n}.
\end{equation}
\end{assp}
Let us begin to establish the fundamental result of this paper that reveals
the boundedness of the second moments for SDEs \eqref{eq:SDE} under Assumptions
\ref{a0} and \ref{as:polynomial}.

\begin{theorem} \label{TL2}
Let Assumptions \ref{a0}, \ref{os_lip}, \ref{as:polynomial} hold, and $\o\ge 0.5$.
Then, for any $T>0$, there exists a constant $C(T)>0$, such that
the $\o$-EM scheme \eqref{oEM} has the following property
\begin{equation*}
\sup_{\D t\le \D t^*} \sup_{0\le t_{k}\le T} \E\lev X_{t_{k}}\rev ^{2}< C(T).
\end{equation*}
\end{theorem}

\begin{proof}
By  definition \eqref{eq:F} of function $F$, we can represent the $\o$-EM scheme (\ref{oEM}) as
\begin{equation*}
F(X_{t_{k+1}})=F(X_{t_{k}})+f(X_{t_{k}})\D t+g(X_{t_{k}})\D w_{t_{k}}.
\end{equation*}
Consequently writing $\<F(X_{t_{k+1}}),F(X_{t_{k+1}})\> = \lev F(X_{t_{k+1}}) \rev^{2}$
and utilizing Assumption \ref{a0} we obtain
\begin{align} \label{eq:Fexp}
\lev F(X_{t_{k+1}})\rev^2 = & \lev F(X_{t_{k}})\rev^2+\lev f(X_{t_{k}})\rev^2\Delta t^2+\lev g(X_{t_{k}})\rev^2\D t\\
& +  2\<F(X_{t_{k}}),f(X_{t_{k}})\>\D t+\D M_{t_{k+1}} \nonumber \\
 = & \lev F(X_{t_{k}})\rev^2 \nonumber\\
& + \left(2\<X_{t_{k}},f(X_{t_{k}})\> +\lev g(X_{t_{k}})\rev^2 \right)\Delta t \nonumber \\
& +  (1-2\o)\lev f(X_{t_{k}})\rev^2\D t^2+\D M_{t_{k+1}}, \nonumber
\end{align}
where
\begin{align*}
\D M_{t_{k+1}} & =  \lev g(X_{t_{k}}) \D w_{t_{k+1}} \rev^{2} - \lev g(X_{t_{k}})\rev^2 \D t
               +2 \< F(X_{t_{k}}) , g(X_{t_{k}}) \Delta w_{t_{k+1}} \> \\
           & +  2\<f(X_{t_{k}})\Delta t,g(X_{t_{k}})\D w_{t_{k+1}}\>,
\end{align*}
is a local martingale. By Assumption \ref{a0} and the fact that $\o\ge0.5$,
\begin{align} \label{eq:ineq}
\lev F(X_{t_{k+1}})\rev^2
 \le & \, \lev F(X_{t_{k}})\rev^2 +  2\a \D t + 2\be \lev X_{t_{k}} \rev^{2}  \D t +\D M_{t_{k+1}}.
\end{align}
Let $N$ be any nonnegative integer such that $N \D t \le T$.  Summing up both sides of
inequality \eqref{eq:ineq} from $k=0$ to $N\we\l_m$, we get
\begin{align} \label{f1}
\lev F( X_{t_{N\we\l_m+1}} ) \rev ^{2}
&\leq
\lev F( X_{t_{0}} )\rev ^{2} +2 \a T
  + 2\be \sum_{k=0}^{N\we\l_m} \lev X_{t_{k}} \rev^{2} \D t
       +\sum_{k=0}^{N\we\l_m} \D M_{t_{k+1}} \notag \\
     & \le
    \lev F( X_{t_{0}} )\rev ^{2} + 2\a T
  + 2\be \sum_{k=0}^{N} \lev X_{t_{k\we \l_{m}}} \rev^{2} \D t
       +\sum_{k=0}^{N} \D M_{t_{k+1}} \1_{[0,\lambda _{m}]}(k).
\end{align}
Applying Lemma \ref{stopping},  Assumption \ref{as:polynomial} and noting that $X_{t_{k}}$ and
$\1_{[0,\l_{m}]}(k)$ are $\F_{t_k}$-measurable while
$\D w_{t_{k}}$ is independent of $\F_{t_k}$, we can take 
expectations on both sides of (\ref{f1}) to get
\begin{align*}
  \E
 \lev F( X_{t_{N\we\l_m+1}} ) \rev^{2} \notag
   &\le
 \lev F( X_{t_{0}} ) \rev^2 + 2\a T + 2\be \, \E
   \left[
   \sum_{k=0}^{N} \lev X_{t_{_{k\we \l_{m}}}} \rev^{2} \D t
   \right].
\end{align*}
By Lemma \ref{lem:F}
\begin{align*}
  \E
 \lev F( X_{t_{N\we\l_m+1}} ) \rev^{2} \notag
   \le  &
  \lev F( X_{t_{0}} ) \rev^2 + (2\a + 2\be( 1- 2 \be \o \D t)^{-1}\,2\o\a\D t ) (T+\D t) \\
& + 2 \be \, ( 1- 2 \be \o \D t)^{-1} \, \E
   \left[
   \sum_{k=0}^{N} \lev F(X(t_{_{k\we \l_{m}}})) \rev^{2} \D t
   \right].
\end{align*}
By the discrete Gronwall inequality
\begin{equation} \label{f1b}
 \E
    \lev F( X_{t_{N\we\l_m+1}} ) \rev^{2}
 \le
  \left[ \lev F( X_{t_{0}} ) \rev^2 +(2\a + 2\be ( 1- 2 \be \o \D t)^{-1}\,2\o\a\D t ) 
(T+\D t) \right] \exp\left(2\be \, ( 1- 2 \be \o \D t)^{-1}(T+\D t)\right),
  \end{equation}
where we use the fact that $N\D t\le T$. Thus, letting $m\rightarrow \infty $ in (\ref{f1b}) and applying
Fatou's lemma, we get
\[
\ \E
    \lev F (X_{t_{N+1}} )\rev^{2}
 \le
   \left[ \lev F( X_{t_{0}} ) \rev^2 +(2\a + 2\be ( 1- 2 \be \o \D t)^{-1}\,2\o\a\D t ) 
(T+\D t) \right] \exp\left(2\be \, ( 1- 2 \be \o \D t)^{-1}(T+\D t)\right).
  \]
By Lemma \ref{lem:F}, the proof is complete.
\end{proof}
\section{Forward-Backward Euler-Maruyama Scheme}
We find in our analysis that it is convenient to work with a continuous extension
of a numerical method. This continuous extension enables us to use the powerful
continuous-time stochastic analysis in order to formulate theorems on numerical
approximations. We find it particularly useful in the proof of forthcoming Theorem
\ref{ProbE}. Let us define
\begin{equation*}
\eta (t):=t_{k},\quad \mathrm{~for~} \quad t\in \lbrack t_{k},t_{k+1}), \ k\ge
0,
\end{equation*}
\begin{equation*}
\eta_{+} (t):=t_{k+1},\quad \mathrm{~for~} \quad t\in \lbrack t_{k},t_{k+1}), \ k\ge
0.
\end{equation*}
One possible  continuous version of the $\o$-EM is given by
\begin{equation} \label{CEMdef}
X(t)=X_{t_0} + \o \, \int_{0}^{t}f(X_{\eta_{+}(s)})ds
   + (1-\o) \, \int_{0}^{t}f(X_{\eta(s)})ds + \int_{0}^{t}g(X_{\eta(s)})dw(s), \quad t\ge 0.
\end{equation}
Unfortunately, this $X(t)$ is not $\mathcal{F}_t$-adapted whence it does not meet
the fundamental requirement in the classical stochastic analysis. 
To avoid using Malliavin calculus, we introduce a new numerical scheme, which
we call the \emph{Forward-Backward Euler-Maruyama (FBEM) scheme}: 
 Once we compute the discrete values $X_{t_{k}}$ from the $\o$-EM,
that is
\begin{equation*}
X_{t_{k}}=X_{t_{k-1}} + \o f(X_{t_{k}})\D t + (1-\o) f(X_{t_{k-1}})\D t
+ g(X_{t_{k-1}})\D w_{t_{k-1}},
\end{equation*}
we define the discrete FBEM by 
\begin{equation} \label{eq:FBEM}
\hat{X}_{t_{k+1}}=\hat{X}_{t_{k}}+f(X_{t_{k}})\D t
+g(X_{t_{k}})\Delta w_{t_{k}},
\end{equation}
where $\hat{X}_{t_{0}}=X_{t_{0}}=x_{0}$, and 
the continuous FBEM   by
\begin{equation} \label{CBFEMdef}
\hat{X}(t)=\hat{X}_{t_0} +\int_{0}^{t}
f(X_{\eta(s)})ds+\int_{0}^{t}g(X_{\eta (s)})dw(s), \quad t\ge 0.
\end{equation}
Note that
the continuous and
discrete BFEM schemes coincide at the gridpoints; that is,
$\hat{X}(t_k) = \hat{X}_{t_k}$.
%
\subsection{Strong Convergence on the Compact Domain}
It this section we prove the strong convergence theorem. We begin by showing that
both schemes of the FBEM \eqref{eq:FBEM} and the $\o$-EM \eqref{oEM} stay
close to each other on a compact domain.
Then we estimate the probability that both
continuous FBEM \eqref{CBFEMdef} and $\o$-EM (\ref{oEM}) will not explode on a finite time interval.
\begin{lemma} \label{BFEMp}
Let Assumptions \ref{a0}, \ref{os_lip}, \ref{as:polynomial} hold, and $\o\ge 0.5$.
Then for any integer $p\ge 2$ and $m\ge \lev x_{0} \rev$, there exists
a constant $C(m,p)$ such that
\begin{equation*}
\E \left[ \lev \hat{X}_{t_{k}}-X_{t_{k}}\rev ^{p}\1_{[0,\l_{m}]}(k)\right] \le C(m,p) \D t^{p}, \qquad \forall k \in \NN,
\end{equation*}
and for $F(x) = x - \o f(x) \D t$ we have
\begin{equation*} 
\lev \hat{X}_{t_{k}} \rev^{2} \ge \frac{1}{2} \lev F(X_{t_{k}})\rev^{2} - \lev \o f(x_{0}) \D t \rev^{2}  \qquad \forall k \in \NN.
\end{equation*}
\end{lemma}
\begin{proof}
Summing up both schemes of the FBEM \eqref{eq:FBEM} and the $\o$-EM \eqref{oEM}, respectively, we obtain
\begin{equation*}
 \hat{X}_{t_{N}} - X_{t_{N}} = \o(   f(X_{t_{0}}) -f(X_{t_{N}})     ) \D t.
\end{equation*}
By H\"{o}lder's inequality, Lemma \ref{stopping} and Assumption \ref{as:polynomial}, we then
see easily that there exists a constant $C(m,p)>0$, such that
\begin{equation} \label{foregrid}
\E
    \left[
         \lev \hat{X}_{t_{N}}-X_{t_{N}}\rev ^{p} \1_{[0,\l_{m}]}(N) \right]
            = \o \,\E
                \left[ \lev   f(X_{t_{0}})\Delta t -f(X_{t_{N}})\D t\rev ^{p} \1_{[0,\l_{m}]}(N) \right]
                \leq C(m,p) \D t ^{p},
\end{equation}
as required. Next, using inequality $2|a||b|\le \e|a|^2 + \e^{-1}|b|^2$ with $\e=0.5$ we arrive at
\begin{align*}
 \lev \hat{X}_{t_{N}} \rev^{2} = &\lev X_{t_{N}} -  \o f(X_{t_{N}}) \D t  + \o   f(X_{t_{0}}) \D t\rev^{2}
\ge (\lev F(X_{t_{N}}) \rev - \lev \o   f(X_{t_{0}}) \D t\rev )^{2} \\
\ge & \frac{1}{2} \lev F(X_{t_{N}}) \rev^{2} - \lev \o   f(X_{t_{0}}) \D t\rev^{2}.
\end{align*}

\end{proof}
The following Theorem provides us with a similar estimate for the distribution of the first passage time
for the continuous FBEM \eqref{CBFEMdef} and $\o$-EM \eqref{oEM} as we have obtained 
for the SDEs \eqref{eq:SDE}
in Theorem \ref{Existence}.
We will use this estimate in the proof of forthcoming Theorem \ref{1dconv}.
\begin{theorem} \label{ProbE}
Let Assumptions \ref{a0}, \ref{os_lip}, \ref{as:polynomial} hold, and $\o \ge 0.5$. Then,  for any
given $\epsilon>0$, there exists a positive integer $N_{0}$  such that for every $m
\geq N_{0}$, we can find a positive number $\D t_0=\D t_0(m)$ so that whenever $\D t
\le \D t_0$,
\begin{equation*}
\PP(\vartheta _{m}<T)\leq \epsilon, \quad \mathrm{for} \,\,\, T>0,
\end{equation*}
where $\vartheta _{m}=\inf \{t>0: \lev \hat{X}(t) \rev \ge m
\  \mathrm{~or~} \ \lev X_{\eta(t)}\rev >m \}$.
\end{theorem}
\begin{proof}
The proof is given in the Appendix.
\end{proof}
\subsection{Strong Convergence on the Whole Domain}
In this section we present the main theorem of this paper, the strong
convergence of the $\o$-EM (\ref{oEM}) to the solution of (\ref{eq:SDE}). First, we will
show that the continuous FBEM (\ref{CBFEMdef}) converges to the true solution on
the compact domain.
This, together with Theorem \ref{ProbE}, will enable us to extend convergence to the whole domain.
Let us define the stopping time
$$\theta _{m}=\tau_{m}\wedge \vartheta_{m},$$
where $\tau_{m}$ and $\vartheta_{m}$ are defined in Theorems \ref{Existence} and \ref{ProbE}, respectively.
\begin{lemma} \label{C1}
Let Assumptions \ref{a0}, \ref{os_lip}, \ref{as:polynomial} hold, and $\o \ge 0.5$. For 
sufficiently large $m$, there exists a positive
constant $C(T,m)$, such that
\begin{equation} \label{eq:compactcon}
\E\left[
\sup_{0\leq t\leq T}\lev \hat{X}(t\wedge\theta _{m})-x(t\wedge \theta _{m})\rev^{2}
\right]
\leq
C(T,m) \D t.
\end{equation}
\end{lemma}
\begin{proof}
The proof is given in the Appendix.
\end{proof}
We are now ready to prove the strong convergence  of the $\o$-EM \eqref{oEM} to the true solution of \eqref{eq:SDE}.

\begin{theorem} \label{1dconv}
Let Assumptions \ref{a0}, \ref{os_lip}, \ref{as:polynomial} hold, and $\o \ge 0.5$. For any given $T=N\,\D t>0$
 and $s\in[1,2)$, $\o$-EM scheme \eqref{oEM} has the property
\begin{equation} \label{eq:error}
\lim_{\D t \rightarrow 0} \E \lev X_{T}-x(T)\rev^{s}=0.
\end{equation}
\end{theorem}

\begin{proof} Let
$$e(T)=X_T-x(T).$$
Applying Young's inequality
\begin{equation*}
x^{s}y\leq \frac{\delta s}{2}x^{2}+\frac{2-s}{2\delta^{\frac{s}{2-s}}}y^{\frac{2}{2-s}}, \quad \forall x,y,\delta >0,
\end{equation*}
leads us to
\begin{eqnarray}
\E\lev e(T)\rev ^{s}
&=&
\E
\left[
\lev e(T)\rev^{s}  \1_{\{\tau _{m}>T,\vartheta _{m}>T\}}
\right]
+\E
\left[
  \lev e(T)\rev^{s} \1_{\{\tau_{m}\leq T\quad or\quad \vartheta _{m}\leq T\}}
\right]  \notag \\
    &\leq &
      2^{s-1}
        \left[
            \E
            [
             \lev \hat{X}(T)-x(T)\rev^{s}\1_{\{\tau _{m}>T,\vartheta _{m}>T\}}
            ]
                +\E
            [
                \lev X_T-\hat{X}(T)\rev^{s}\1_{\{\tau _{m}>T,\vartheta _{m}>T\}}
            ]
        \right] \notag \\
    & + &
        \frac{\delta s}{2}\E
            \left[
             \lev e(T)\rev ^{2}
            \right]
                 + \frac{2-s}{2\delta^{\frac{s}{2-s}}}
                 \PP(\tau _{m}
                    \le
                    T\quad \mathrm{or}\quad \vartheta _{m}\leq T). \notag
                    \end{eqnarray}
First, let us observe that by Lemma \ref{BFEMp} we obtain
\begin{equation*}
\E[\lev X_T-\hat{X}(T)\rev^{s}\1_{\{\tau _{m}>T,\vartheta _{m}>T\}}]
\le
C(m,s)\D t^{s}.
\end{equation*}
Given an $\epsilon >0$, by H\"{o}lder's inequality and Theorems \ref{Existence} and \ref{TL2}  , we
choose $\delta$ such that
\begin{equation*}
\frac{\delta s}{2} \E
\left[
\lev e(T)\rev ^{2}
\right]
\leq
\delta s
\E
    \left[
      \lev x(T)\rev^{2}
     +\lev X_T \rev^{2}
    \right]
    \leq
    \frac{\epsilon}{3}.
\end{equation*}
Now by (\ref{eq:lim}) there exists $N_{0}$ such that for
$m\geq N_{0}$
\begin{equation*}
\frac{2-s}{2\delta^{\frac{s}{2-s}}}
    \PP(\tau _{m}\leq T)
        \leq
            \frac{\epsilon}{3},
\end{equation*}
and finally by Lemmas \ref{BFEMp}, \ref{C1} and Theorem \ref{ProbE} we choose
$\D t$ sufficiently small, such that
\begin{eqnarray*}
2^{s-1}
\left[
    \E[\lev \hat{X}(T)-x(T)\rev^{s}\1_{\{\tau _{m}>T,\vartheta _{m}>T\}}]
        +\E[\lev X_T-\hat{X}(T)\rev^{s}\1_{\{\tau _{m}>T,\vartheta _{m}>T\}}]
\right]
    +
    \frac{2-s}{2\delta^{\frac{s}{2-s}}}
    \PP(\vartheta _{m}\leq T)
         \le
            \frac{\epsilon}{3},
\end{eqnarray*}
which completes the proof.
\end{proof}
Theorem \ref{1dconv} covers many highly non-linear SDEs, though it might be computationally expensive to find the
inverse $F^{-1}$ of the function $F(x)=x-\o f(x)\D t$.
For example, lets consider equation \eqref{eq:ex} with $\mu(x) = a + \sin(x)^{2}$, $a>0$,  that is
\begin{equation} \label{eq:electricity}
 dx(t) = ( a + \sin(x)^{2} - \a x(t)^3)dt + \be x(t)^{2}dw(t), 
\end{equation}
where $\a,\be>0$. This type of SDE could be used to model electricity prices where we need to account for
a seasonality pattern, \cite{lucia2002electricity}. 
In this situation, it is useful to split the drift coefficient
in two parts, that is
\begin{equation}
f(x)=f_{1}(x)+f_{2}(x).
\end{equation}
This allows us to introduce partial implicitness in the numerical scheme. 
In the case of \eqref{eq:electricity} we would take $f_{1}(x)= - \a x(t)^3$ and $f_{2}(x) = a + \sin(x)^{2}$.
Then a new partially implicit $\o$-EM scheme has the following form
\begin{equation} \label{eq:NewBEM}
 Y_{t_{k+1}} = Y_{t_{k}}+\o f_{1}(Y_{t_{k+1}}) \D t+(1-\o) f_{1}(Y_{t_{k}}) \D t
 + f_{2}(Y_{t_{k}})\D t+g(Y_{t_{k}})\Delta w_{t_{k}}.
\end{equation}
It is enough that $f_{1}$ satisfies Assumption \ref{os_lip} in order for scheme  \eqref{eq:NewBEM}
to be well defined. Its solution can be represented as
\[
 Y_{t_{k+1}} = H^{-1} \left(  Y_{t_{k}} + (1-\o) f_{1}(Y_{t_{k}}) \D t
 + f_{2}(Y_{t_{k}})\D t+g(Y_{t_{k}})\Delta w_{t_{k}}   \right),
\]
where
\begin{equation} \label{eq:F1}
 H(x) = x - \o f_{1}(x) \D t,
\end{equation}
All results from Sections 3 and 4 hold,
 once we replace condition (\ref{as1}) in Assumption
\ref{a0} and Assumption \ref{os_lip}  by  \eqref{eq:f1Lip} 
\eqref{con2}), respectively. 

\begin{theorem} \label{th:newH}
Let Assumption \ref{as:polynomial} hold and $\D t\in[0,(\max\{L,2\be\}\o)^{-1})$. In addition we assume that
for $x,y \in \RR^{n}$, there exist constants $L,\a,\be>0$ such that
\begin{equation} \label{eq:f1Lip}
\<x-y , f_{1}(x)-f_{1}(y) \>\le L \lev x-y\rev^{2}, 
\end{equation}
and
\begin{align} \label{con2}
&\<x,f(x)\> +\frac{1}{2}\lev g(x)\rev^{2} + \bigl[ (1-\o)\<f_{1}(x),f_{2}(x)\>
 + \frac{1}{2}\lev f_{2}(x)\rev^{2} + \frac{1}{2}(1-2\o)\lev f_{1}(x) \rev^{2}   \bigr] \D t 
 \le \a+\be\lev x\rev^{2}.
\end{align}
Then for any given $T>0$ and $s\in[1,2)$, $\o$-EM scheme \eqref{eq:NewBEM} has the following property
\begin{equation} \label{eq:error}
\lim_{\D t \rightarrow 0} \E \lev Y_{T}-x(T)\rev^{s}=0.
\end{equation}
\end{theorem}

\begin{proof}
In order to prove Theorem \ref{th:newH} we need to show that results from sections 3 and 4, proved for \eqref{oEM},
hold for \eqref{eq:NewBEM} under modified assumptions. The only significant difference is in the proof of
Theorem \ref{TL2} for \eqref{eq:NewBEM}.  
Due to condition \eqref{eq:f1Lip} we can show that Lemma \ref{lem:F} holds for function $H$.
Then by the definition of function $H$ in \eqref{eq:F1}, we can represent the $\o$-EM scheme \eqref{eq:NewBEM} as
\begin{equation*}
H(Y_{t_{k+1}}) = H(Y_{t_{k}})+f(Y_{t_{k}})\D t+g(Y_{t_{k}})\D w_{t_{k}}.
\end{equation*}
Consequently writing $\<H(Y_{t_{k+1}}),H(Y_{t_{k+1}})\> = \lev H(Y_{t_{k+1}}) \rev^{2}$ we obtain
\begin{align} \label{Eq:MS}
\lev H(Y_{t_{k+1}})\rev^2& =  \lev H(Y_{t_{k}})\rev^2+\lev f(Y_{t_{k}})\rev^2\Delta t^2+\lev g(Y_{t_{k}})\rev^2\D t\\
& +  2\<H(Y_{t_{k}}),f(Y_{t_{k}})\>\D t+\D M_{t_{k+1}} \nonumber \\
& = \lev H(Y_{t_{k}})\rev^2
 + \left(2\<Y_{t_{k}},f(Y_{t_{k}})\> +\lev g(Y_{t_{k}})\rev^2 \right)\D t \nonumber \\
& + \left[ 2( 1 - \o )\< f_{1}(Y_{t_{k}}), f_{2}(Y_{t_{k}})\>  +  \lev f_{2}(Y_{t_{k}}) \rev^{2}
  + (1-2\o)\lev f_{1}(Y_{t_{k}})\rev^2   \right] \D t^{2}
  + \D M_{t_{k+1}} \nonumber ,
\end{align}
where
\begin{align*}
\D M_{t_{k+1}} & =  \lev g(Y_{t_{k}}) \D w_{t_{k+1}} \rev^{2} - \lev g(Y_{t_{k}})\rev^2 \D t
               +2 \< H(Y_{t_{k}}) , g(Y_{t_{k}}) \Delta w_{t_{k+1}} \> \\
           & +  2\<f(Y_{t_{k}})\D t,g(Y_{t_{k}})\D w_{t_{k+1}}\>.
\end{align*}
Due to condition \eqref{con2} we have
\begin{align*}
\lev H(Y_{t_{k+1}})\rev^2
& \le \lev H(Y_{t_{k}})\rev^2
 + 2 \a \D t + 2 \be \, \lev Y_{t_{k}} \rev^{2}
  + \D M_{t_{k+1}}.
\end{align*}
The proof can be completed by analogy to the analysis for the $\o$-EM scheme \eqref{oEM}.
Having boundedness of moments for \eqref{eq:NewBEM} we can show that \eqref{eq:error} holds
in exactly the same way as for $\o$-EM scheme \eqref{oEM}.
\end{proof}
\subsection{Numerical Example}
In this section we perform a numerical experiment that confirms our theoretical results.
Since Multilevel Monte-Carlo simulations provide excellent motivation for our work \cite{giles2008multilevel},
here we consider the measure of error \eqref{eq:error} with $s=2$. Although, the case $s=2$
is not covered by our analysis, the numerical experiment suggests that Theorem \ref{1dconv}
still holds.
In our numerical experiment, we focus on
the error at the endpoint $T=1$, so we let
\begin{equation*}
e^{strong}_{\D t}=\E \lev x(T)-X_{T}\rev^{2}.
\end{equation*}
We consider the SDE (\ref{eq:ex})
\begin{equation*}
dx(t)=(\mu -\a x(t)^{3})dt + \be x^{2}(t)dw(t),
\end{equation*}
where
$(\mu,\a,\be) = (0.5,0.2,\sqrt{0.2})$. The assumptions
of Theorem \ref{1dconv} hold. The $\o$-EM \eqref{oEM} with $\o=1$, applied to \eqref{eq:ex} writes as
\begin{equation} \label{eq:BEMinterest}
X_{t_{k+1}}=X_{t_{k}}+(\mu-\a X_{t_{k+1}}^{3})\D t + \be X_{t_{k}}^{2}\D w_{t_{k}}.
\end{equation}
Since we employ the BEM to approximate \eqref{eq:ex},
on each step of the numerical simulation we need to find the inverse of the
function $F(x)= \a x^{3}\D t+x$. In this case we can find the inverse function
explicitly and therefore computational
complexity is not increased. Indeed, we observe that it is enough to
find the real root of the cubic equation
\begin{equation*} 
\a X^{3}_{t_{k+1}}\D t+X_{t_{k+1}}-(X_{t_{k}}
+\mu\D t+\be X_{t_{k}}^{2} \D w_{t_{k}}) = 0.
\end{equation*}
In Figure \ref{fig:stronginterest} we plot $e^{strong}_{\D t}$ against
 $\D t$ on a log-log scale. Error bars representing $95\%$ confidence intervals
 are denoted by circles.
\begin{figure}[htb]
\begin{center}
\includegraphics[scale=0.8]{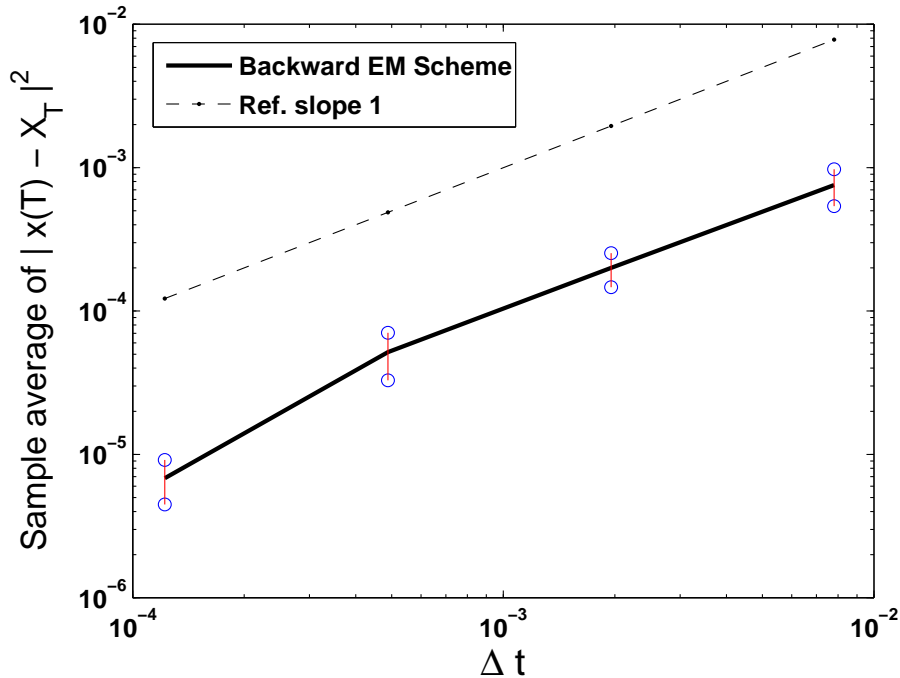}
\end{center}
        \caption{A strong error plot where the dashed line is a reference slope and the continuous line
           is the extrapolation of the error estimates for the BEM scheme.}

\label{fig:stronginterest}
\end{figure}
Although we do not know the explicit form of the solution to
\eqref{eq:ex}, Theorem \ref{1dconv} guarantees
that the BEM \eqref{eq:BEMinterest} strongly converges to the true solution. Therefore, it
is reasonable to take the BEM with the very small time step $\D t = 2^{-14}$
as a reference solution. We then compare it to the BEM evaluated with timesteps
$( 2^{1}\D t, 2^{3}\D t, 2^{5}\D t, 2^{7}\D t)$
in order to estimate the rate of convergence. Since we are using Monte Carlo method,
the sampling error decays with a rate of $1/\sqrt{M}$, $M$- is the number of sample paths.
We set $M=10000$.
From Figure \ref{fig:stronginterest} we see that there appears to exist a positive constant
such that
\begin{equation*}
e^{strong}_{\D t}\le C \D t \quad \hbox{for sufficiently small $\D t$}.
\end{equation*}
Hence, our results are consistent
with a strong order of convergence of one-half.
\section{Stability Analysis}
In this section we examine the globally almost surely asymptotic stability   of the $\o$-EM scheme \eqref{oEM}.
The stability conditions we derive are more related to the mean-square stability, 
\cite{higham2008almost,mattingly2002ergodicity}.
First, we give some preliminary analysis for the SDEs \eqref{eq:SDE}. We give conditions on the coefficients
of the SDEs (\ref{eq:SDE}) that are sufficient for the globally almost surely asymptotic stability.
Later we prove that the $\o$-EM scheme (\ref{oEM}) reproduces this asymptotic behavior very well.
\subsection{Continuous Case}
Here we present a simplified version of the stochastic LaSalle Theorem as proved in \cite{shen2006improved},
using the Lyapunov function $V(x)=\lev x\rev^{2}$.

\begin{theorem}[Mao et al.]\label{T2}
Let Assumption \ref{a0} hold. We assume that there exists a function $z \in C(\RR^{n};\RR_{+})$, such that
\begin{equation}
\<x,f(x)\> + \frac{1}{2}\lev g(x) \rev^{2}\le -z(x)  \label{c1}
\end{equation}
for all $x\in \RR^{n}$.
We then have the following assertions:
\begin{itemize}
  \item For any $x_{0}\in \RR^{n}$, the solution $x(t)$ of (\ref{eq:SDE}) has the properties that
\begin{equation*}
\limsup_{t \rightarrow \8}\lev x(t)\rev^{2}<\8 \quad \hbox{a.s} \quad \hbox{and}
\end{equation*}
\begin{equation*}
\lim_{t \rightarrow \8}z(x(t))=0 \quad \hbox{a.s}.
\end{equation*}
\end{itemize}
What is more, when $z(x)=0$ if and only if $x=0$ then
\begin{equation*}
\lim_{t \rightarrow \8}x(t)=0 \quad \hbox{a.s} \quad \hbox{$\forall$ $x_{0}\in\RR^{n}$}.
\end{equation*}
\end{theorem}
%
\subsection{Almost Sure Stability}
We begin this section with the following Lemma.
\begin{lemma} \label{L2}
Let $Z=\{Z_{n}\}_{n\in\NN}$ be a nonnegative stochastic
process with the Doob decomposition
${Z_{n}}=Z_{0}+A_{n}^{1}-A_{n}^{2}+M_{n}$, where
$A^{1}=\{A_{n}^{1}\}_{n\in\NN}$ and $A^{2}=\{A_{n}^{2}\}_{n\in\NN}$
are a.s. nondecreasing, predictable processes with
$A_{0}^{1}=A_{0}^{2}=0$, and $M=\{M_{n}\}_{n\in\NN}$ is local
$\{ \F_{n} \}_{n\in\NN}$-martingale with $M_{0}=0$.
Then
\begin{equation*}
\left\{ \lim_{n \rightarrow \8} A^{1}_{n}<\infty \right\} \subseteq \left\{
\lim_{n \rightarrow \8} A^{2}_{n}<\infty \right\} \cap \left\{
\lim_{n \rightarrow \8} Z_{n}\hbox{ exists and is finite} \right\} \quad \hbox{a.s.}
\end{equation*}
\end{lemma}
The original lemma can be found in Liptser and Shiryaev
\cite{liptser1989theory}. The reader can notice that this lemma combines
the Doob decomposition and the martingales convergence theorem.
Since we use the Lyapunov function $V(x)=\lev x \rev^2$, our results
extend the mean-square stability
for linear systems, Higham \cite{higham2000stability,higham2001mean},
to a highly non-linear setting.
The next theorem demonstrates that there exists a discrete counterpart of Theorem
\ref{T2} for the $\o$-EM scheme \eqref{oEM}.
\begin{theorem} \label{stability}
Let Assumptions \ref{a0}, \ref{os_lip} and \ref{as:polynomial} hold. Assume that there exists a function  $z \in C(\RR^{n};\RR_{+})$ such that
for all $x\in\RR^n$ and for all $\D t\in(0,(\max\{L,2\be\}\o)^{-1})$,
\begin{equation} \label{stab-EM}
\<x,f(x)\>+\frac{1}{2}\lev g(x)\rev
^{2}+\frac{(1-2\o)}{2}\lev f(x) \rev^{2} \D t\leq -z( x).
\end{equation}
Then the $\o$-EM solution defined by (\ref{oEM}), obeys
\begin{equation}
\limsup_{k\rightarrow \infty }\lev X_{t_{k}}\rev^{2}<\8 \quad \hbox{a.s.,}  \label{cn3}
\end{equation}
\begin{equation} \label{eq:wbound}
\lim_{k\rightarrow \infty }z\left( X_{t_{k}}\right) =0 \text{ a.s.}
\end{equation}
If additionally $z(x)=0$ iff $x=0$, then
\begin{equation} \label{eq:Ostability}
\lim_{k\rightarrow \infty }  X_{t_{k}} =0\text{ a.s.}
\end{equation}
\end{theorem}

\begin{proof}
By \eqref{eq:Fexp} we have
\begin{align*} 
\lev F(X_{t_{k+1}})\rev^2 = &  \lev F(X_{t_{k}}) \rev^{2}
+ \left(2\<X_{t_{k}},f(X_{t_{k}})\> +\lev g(X_{t_{k}})\rev^2 \right)\Delta t \nonumber \\
& +  (1-2\o)\lev f(X_{t_{k}})\rev^2\D t^2+\D M_{t_{k+1}} \nonumber ,
\end{align*}
where
\begin{align*}
\D M_{t_{k+1}} & =  \lev g(X_{t_{k}}) \D w_{t_{k+1}} \rev^{2} - \lev g(X_{t_{k}})\rev^2 \D t
               +2 \< F(X_{t_{k}}) , g(X_{t_{k}}) \Delta w_{t_{k+1}} \> \\
           & +  2\<f(X_{t_{k}})\Delta t,g(X_{t_{k}})\D w_{t_{k+1}}\>,
\end{align*}
so $\sum_{k=0}^{N}\D M_{t_{k+1}}$ is a local martingale
due to Assumption \ref{as:polynomial} and Lemma \ref{stopping}.
Hence, we have obtained the decomposition required to apply Lemma \ref{L2}, that is
\begin{equation*}
\lev F(X_{t_{N+1}})\rev ^{2}=\lev F(X_{t_{0}})\rev^2-\sum_{k=0}^{N}A_{t_{k}}\D t+\sum_{k=0}^{N}\D M_{k}.
\end{equation*}
where
\begin{equation} \label{eq: A}
A_{t_{k}} = -\left( \left( 2\<X_{t_{k}},f(X_{t_{k}})\> +\lev g(X_{t_{k}})\rev^2 \right)
+(1-2\o)\lev f(X_{t_{k}})\rev^2\D t \right).
\end{equation}
By condition \eqref{stab-EM}, $\sum_{k=0}^{N}A_{t_{k}}\D t$ is nondecreasing. Hence by Lemma \ref{L2} we arrive at
\begin{equation} \label{im2}
\lim_{k\rightarrow \infty }\lev F(X_{t_{k}})\rev ^{2}<\infty.
\end{equation}
Consequently, by Lemma \ref{lem:F} we arrive at
\begin{equation*}
\limsup_{k\rightarrow \infty }\lev  X(t_{k})\rev ^{2}<\8 \quad \hbox{a.s.}.
\end{equation*}
By Lemma \ref{L2},
\begin{align*}
\sum_{k=0}^{\8} z(X_{t_{k}}) \D t\le\sum_{k=0}^{\8}A_{t_{k}}\D t <\8 \quad\text{a.s},
\end{align*}
which implies
\begin{equation*}
\lim_{k\rightarrow \8}z( X_{t_{k}})=0 \quad\text{a.s}.
\end{equation*}
This completes the proof of \eqref{eq:wbound} and \eqref{eq:Ostability}.
\end{proof}
\appendix
\section{Proof of Lemma~\ref{stopping}}
\begin{proof}
By \eqref{eq:ineq} we obtain 
\begin{align}
\lev F(X_{t_{k}})\rev^2
 \le & \lev F(X_{t_{k-1}})\rev^2 +  2 \a \D t + 2 \be \lev X_{t_{k-1}}  \rev^{2}  \D t +  \D M_{t_{k}} , \nonumber
\end{align}
where
\begin{align*}
\D M_{t_{k}} & =  \lev g(X_{t_{k-1}}) \D w_{t_{k}} \rev^{2} - \lev g(X_{t_{k-1}})\rev^2 \D t
               +2 \< F(X_{t_{k-1}}) , g(X_{t_{k-1}}) \Delta w_{t_{k}} \> \\
           & +  2\<f(X_{t_{k-1}})\D t,g(X_{t_{k-1}})\D w_{t_{k}}\>.
\end{align*}
Using the basic inequality $(a_{1} + a_{2} + a_{3}+ a_{4})^{p/2}
\le 4^{p/2-1} ( a_{1}^{p/2} + a_{2}^{p/2} + a_{3}^{p/2} + a_{4}^{p/2}   ) $,
where $a_{i}\ge0$, we obtain
\begin{align}
\lev F(X_{t_{k}})\rev^p
 \le & 4^{p/2-1} \left( \lev F(X_{t_{k-1}})\rev^p +  (2 \a \D t)^{p/2}
 + (2\be)^{p/2} \lev X_{t_{k-1}} \rev^{p}  \D t^{p/2} + \mid \D M_{t_{k}} \mid^{p/2} \right).
\end{align}
As a consequence
\begin{align*}
\E \left[ \lev F(X_{t_{k}})\rev^p \1_{[0,\l_{m}]}(k) \right]
 \le & 4^{p/2-1} \biggl( \E \left[ \lev F(X_{t_{k-1}})\rev^p \1_{[0,\l_{m}]}(k) \right]  +  (2\a\D t )^{p/2}
+ (2\be m^{2} \D t)^{p/2} \\
 &  + \E \left[ \mid \D M_{t_{k}} \mid^{p/2} \1_{[0,\l_{m}]}(k) \right] \biggr).
\end{align*}
Due to Assumption \ref{as1} we can bound $\lev F(x)\rev^p$ and $\lev g(x)\rev$ for $\lev x\rev<m$.
Whence, there exists a constant $C(m,p)$, such that
\begin{align*}
& \E \biggl[ \mid \D M_{t_{k}} \mid^{p/2}  \biggr] \1_{[0,\l_{m}]}(k) \\
& \le 4^{p/2-1} \E \biggl[  \lev g(X_{t_{k-1}}) \D w_{k} \rev^{p} +\lev g(X_{t_{k-1}})\rev^p \D t^{p/2}
               + ( 2 \lev F(X_{t_{k-1}}) \rev \lev g(X_{t_{k-1}}) \Delta w_{t_{k}} \rev )^{p/2} \\
           & +  ( 2 \lev f(X_{t_{k-1}})\D t \rev \lev g(X_{t_{k-1}})\D w_{t_{k}} \rev )^{p/2}   \biggr] \1_{[0,\l_{m}]}(k) \\
& \le C(m,p) \E \biggl[ 1 +  \lev g(X(t_{k-1})) \D w_{k} \rev^{p} \biggr] \1_{[0,\l_{m}]}(k),
\end{align*}
By H\"{o}lder's inequality
\begin{align*}
& \E \biggl[ \mid \D M_{t_{k}} \mid^{p/2}  \biggr] \1_{[0,\l_{m}]}(k) \\
& \le C(m,p)  \biggl[ 1 +  \bigl(\E\bigl[\lev g(X_{t_{k-1}}) \rev^{2p}\1_{[0,\l_{m}]}(k)\bigr]\bigr)^{1/2}
	\bigl(\E\lev\Delta w_{t_{k-1}}\rev^{2p}\bigr)^{1/2}
           \biggr].
\end{align*}
Hence
\begin{align*}
\E \left[\lev F(X_{t_{k}})\rev^{p}\1_{[0,\l_{m}]}(k)\right]
& \leq
C(m,p) \biggl( 1 
        +\bigl(\E\bigl[\lev g(X_{t_{k-1}}) \rev^{p}\1_{[0,\l_{m}]}(k)\bigr]^{2}\bigr)^{1/2}
	\bigl(\E\lev\Delta w_{t_{k-1}}\rev^{2p}\bigr)^{1/2} \biggr)
\end{align*}
Since there exists a positive constant $C(p)$,
such that $\E\lev\Delta w_{t_{k-1}}\rev^{2p}<C(p)$, we obtain
\begin{align*}
\E\left[\lev F (X_{t_{k}} ) \rev ^{p}\1_{[0,\l_{m}]}(k)\right]<C(m,p).
\end{align*}
We conclude the assertion by applying Lemma \ref{lem:F}.
\end{proof}
\section{Proof of Theorem~\ref{ProbE}}
\begin{proof} 
By the It\^{o} formula
\begin{align*}
\lev \hat{X}(T \wedge \vartheta _{m}) \rev^{2}
& = \lev x_{0} \rev^{2} +
\int_{0}^{T \wedge \vartheta _{m}}
\left( 2\< \hat{X}(s),f(X_{\eta(s)}) \>  + \trace[g^{T}(X_{\eta(s)})I_{n\times n}g(X_{\eta(s)})]\right) ds\\
 &   + 2 \int_{0}^{T \wedge \vartheta _{m}}  \<\hat{X}(s),g(X_{\eta(s)})\>dw(s) \\
& = \lev x_{0} \rev^{2} +
\int_{0}^{T \wedge \vartheta _{m}}
\left( 2\< \hat{X}(s)-X_{\eta(s)}+X_{\eta(s)},f(X_{\eta(s)}) \>
 + \lev g(X_{\eta(s)}) \rev^{2} \right) ds\\
 &   + 2 \int_{0}^{T \wedge \vartheta _{m}}  \<\hat{X}(s),g(X_{\eta(s)})\>dw(s) \\
& \le \lev x_{0} \rev^{2} +
\int_{0}^{T \wedge \vartheta _{m}}
\left( 2\< X_{\eta(s)},f(X_{\eta(s)}) \>
 + \lev g(X_{\eta(s)}) \rev^{2} \right) ds \\
  & +  2 \int_{0}^{T \wedge \vartheta _{m}} \lev  \hat{X}(s)-X_{\eta(s)} \rev \lev f(X_{\eta(s)}) \rev  ds \\
 &   + 2 \int_{0}^{T \wedge \vartheta _{m}}  \<\hat{X}(s),g(X_{\eta(s)})\>dw(s).
\end{align*}
By Assumption \ref{a0}, for $\lev x\rev \leq m$
\begin{equation} \label{eq:lipf}
\lev f(x) \rev^{2} \le 2 (\lev f(x)-f(0) \rev^{2}+\lev f(0)\rev^{2})\le 2(C(m)\lev x \rev^{2}+\lev f(0)\rev^{2}),
\end{equation}
\begin{equation} \label{eq:lipg}
\lev g(x) \rev^{2} \le 2 (\lev g(x)-g(0) \rev^{2}+\lev g(0)\rev^{2})\le 2(C(m)\lev x \rev^{2}+\lev g(0)\rev^{2}),
\end{equation}
and
\begin{align*}
\E
  \lev \hat{X}(T\wedge \vartheta _{m}) \rev^{2}
& \leq
 \lev x_{0} \rev^{2}+2\a T
  +2 \be\,\E \int_{0}^{T\we \vartheta _{m}} \lev X_{\eta(s)} -\hat{X}(s) + \hat{X}(s) \rev^{2} ds
  +C(m)\E\int_{0}^{T\wedge \vartheta_{m}}\lev X_{\eta(s)}-\hat{X}(s)\rev ds.
\end{align*}
Using the basic inequality $(a-b+c)^{2}\le 2(\lev a-b \rev^{2} + \lev c \rev^2)$ and 
the fact that $\int_{0}^{T\wedge \vartheta_{m}}\lev X_{\eta(s)}-\hat{X}(s)\rev^{2} ds \le
C(m)\int_{0}^{T\wedge \vartheta_{m}}\lev X_{\eta(s)}-\hat{X}(s)\rev ds$,
we obtain
\begin{align} \label{eq:probe1}
\E
  \lev \hat{X}(T\wedge \vartheta _{m}) \rev^{2}
 \leq & \,
 \lev x_{0} \rev^{2}+2\a T
  +4\be \,\int_{0}^{T}\E \lev \hat{X}(s\wedge \vartheta _{m}) \rev^{2} ds   \notag \\
& + 4 \be \E\int_{0}^{T\wedge \vartheta_{m}}\lev X_{\eta(s)}-\hat{X}(s)\rev^{2} ds
 +C(m)\E\int_{0}^{T\wedge \vartheta_{m}}\lev X_{\eta(s)}-\hat{X}(s)\rev ds \notag \\
 \leq & \,
 \lev x_{0} \rev^{2}+2\a T
  +4\be \int_{0}^{T}\E \lev \hat{X}(s\wedge \vartheta _{m}) \rev^{2} ds
  +C(m)(4\be +1)\E\int_{0}^{T\wedge \vartheta_{m}}\lev X_{\eta(s)}-\hat{X}(s)\rev ds.
\end{align}
Since $\l_{m} \ge \vartheta_{m}$ a.s., Lemma \ref{BFEMp} gives the following bound
\begin{equation} \label{eq:hathat}
\E \int_{0}^{T\wedge \vartheta_{m}}\lev X_{\eta(s)}-\hat{X}_{\eta(s)}\rev ds
\le C(m,T)\D t .
\end{equation}
To bound the term
$\E \int_{0}^{T\wedge \vartheta_{m}}\lev \hat{X}_{\eta(s)}-\hat{X}(s)\rev ds$ in \eqref{eq:probe1},
 first we observe that
\[
\lev \hat{X}_{\eta(s)} -\hat{X}(s)\rev \1_{[t_{k},t_{k+1})}(s)
= \lev \int_{t_{k}}^{s}f(X_{t_{k}})dh+ \int_{t_{k}}^{s}g(X_{t_{k}})dw(h) \rev \1_{[t_{k},t_{k+1})}(s).
\]
Then, by \eqref{eq:lipf} and \eqref{eq:lipg}
\begin{equation} \label{eq:hat2}
\E \int_{0}^{T\wedge \vartheta _{m}}  \lev \hat{X}_{\eta(s)} -\hat{X}(s)\rev  ds \leq C(m,T)\D t^{\frac{1}{2}},
\end{equation}
where $C(m,T)>0$ is constant. Combining \eqref{eq:hathat} and \eqref{eq:hat2} leads us to
\begin{align}
E\int_{0}^{T\wedge \vartheta_{m}}\lev X_{\eta(s)}-\hat{X}(s)\rev ds
    & \le
      \E \int_{0}^{T\wedge \vartheta_{m}}   \lev \hat{X}_{\eta(s)} -\hat{X}(s)\rev ds  \notag\\
    & + \E \int_{0}^{T\wedge \vartheta_{m}} \lev X_{\eta(s)}-\hat{X}_{\eta(s)}\rev ds \notag \\
& \le
C(m,T)\D t^{\frac{1}{2}}.  \label{eq:numbound}
\end{align}
Therefore
\begin{equation*}
\E \lev \hat{X}(T\wedge \vartheta _{m})\rev^{2}
\leq
\lev x_{0} \rev^{2} + 2\a T + C(m,T)\D t^{\frac{1}{2}}
    +4\be\int_{0}^{T}\E \lev \hat{X}(s\wedge \vartheta _{m}) \rev^{2} ds.
\end{equation*}
By Gronwall's inequality
\begin{equation} \label{eq:hattoEM}
\E \lev \hat{X}(T\wedge \vartheta _{m})\rev^{2}
\leq
[\lev x_{0} \rev^{2} + 2\a T + C(m,T)\D t^{\frac{1}{2}}]\exp(4\be T ).
\end{equation}
Now we will find the lower bound for $\lev \hat{X}(\vartheta _{m})\rev^{2}$.
From the definition of the stopping time $\vartheta _{m}$, if  $\inf \{t>0: \lev \hat{X}(t) \rev \ge m\} \le \inf \{t>0:  \lev X_{\eta(t)}\rev >m \}$
then $\lev \hat{X}(\vartheta _{m})\rev^{2} = m^2$. In the alternative case, where 
$\inf \{t>0: \lev \hat{X}(t) \rev \ge m\} > \inf \{t>0:  \lev X_{\eta(t)}\rev >m \}$, we have
$ \lev X_{\vartheta _{m}}\rev^2 >m^2$. From Lemmas \ref{lem:F} and \ref{BFEMp} we arrive at
\[
 \lev \hat{X}(\vartheta _{m}) \rev^{2} \ge \frac{1}{2}\left( ( 1 - 2 \be \o \D t ) \lev X_{\vartheta _{m}} \rev^{2} 
- 2 \o \a \D t \right) - \lev \o f(x_{0})\D t \rev^{2}. 
\]
Hence there exist positive constants $c_{1}$ and $c_{2}$ such that
\[
 \lev \hat{X}(\vartheta _{m}) \rev^{2} > c_{1} m^{2} - c_{2}\D t. 
\]
We have
\[
 \E \lev \hat{X}(T\wedge \vartheta _{m})\rev^{2}
 \ge \E \left[ \1_{\{\vartheta _{m}<T\}} \lev \hat{X}(\vartheta _{m})\rev^{2} \right]
\ge  \PP(\vartheta _{m}<T) ( c_{1} m^{2} - c_{2}\D t ). 
\]

which implies that
\begin{equation*}
\PP(\vartheta _{m}<T)
\le
\frac{ [ \lev x_{0} \rev^{2} + 2\a T + C(m,T)\D t^{1/2}]\exp(4\be T )}  { c_{1} m^{2} - c_{2}\D t}.
\end{equation*}
Now, for any given $\epsilon>0$, we choose $N_{0}$ such that for any $m\ge N_{0}$
\begin{equation*}
\frac{[\lev x_{0} \rev^{2} + 2\a T]\exp(4\be T ) }{c_{1} m^{2} - c_{2}\D t}\le \frac{\epsilon}{2}.
\end{equation*}
Then, we can choose $\D t_{0}=\D t_{0}(m)$, such that for any $\D t\le\D t_{0}$
\begin{equation*}
\frac{\exp(4\be T )C(m,T)\D t^{1/2}}{c_{1} m^{2} - c_{2}\D t}
\leq
\frac{\epsilon}{2},
\end{equation*}
whence $\PP(\vartheta _{m}<T)\le \epsilon$ as required.
\end{proof}
\section{Proof of Lemma~\ref{C1}}
\begin{proof}
It is useful to observe that since the constant $C(T,m)$ in \eqref{eq:compactcon} depends on $m$, we can
prove the theorem in a similar way as in the classical case where coefficients $f$ and $g$ in \eqref{eq:SDE}
obey the global Lipschitz condition \cite{higham2003strong,kloeden1992numerical}. Nevertheless, for completeness of
the exposition we present the sketch of the proof.
For any $T_{1}\in[0,T]$, by H\"{o}lder's and
Burkholder-Davis-Gundy's inequalities
\begin{align*}
&\E
\left[
\sup_{0\leq t\leq T_{1}}\lev
\hat{X}(t\wedge\theta_{m})-x(t\wedge \theta _{m})\rev^{2}
\right]  \\
  &\leq
   2
   \Biggl(
    T \E\int_{0}^{T_{1}\wedge \theta_{m}} \lev f(X_{\eta(s)})-f(x(s)) \rev^{2}ds
    + 4 \E\int_{0}^{T_{1}\wedge \theta _{m}} \lev g(X_{\eta(s)})-g(x(s)) \rev^{2}ds
   \Biggr),
\end{align*}
By Assumption \ref{a0} there exists a constant $C(m)$
\begin{align*}
&\E
\left[
\sup_{0\leq t\leq T_{1}}\lev\hat{X}(t\wedge\theta_{m})-x(t\wedge \theta _{m})\rev^{2}
\right]  \\
&\leq
2 C(m) \Biggl(
T \, \E\int_{0}^{T_{1}\wedge\theta _{m}}\lev X_{\eta(s)}-x(s)\rev^{2} ds
+ 4\,\E \int_{0}^{T_{1}\wedge \theta _{m}} \lev X_{\eta(s)}-x(s)\rev^{2} ds \Biggr)\\
&\leq
4 C(m)
\Biggl(
T
  \E\int_{0}^{T_{1}\wedge\theta _{m}} \left[ \lev \hat{X}(s)-x(s)\rev^{2}
+\lev X_{\eta(s)}-\hat{X}(s) \rev^{2} \right] ds\\
  & +
    4\,\E\int_{0}^{T_{1}\wedge\theta _{m}}\left[ \lev \hat{X}(s)-x(s)\rev^{2}
  + \lev X_{\eta}(s)-\hat{X}(s) \rev ^{2} \right] ds \Biggr)\\
     &\leq
     4 C(m) ( T+4 )
      \E\int_{0}^{T_{1}}\lev \hat{X}(s\wedge\theta_{m})-x(s\wedge \theta _{m})\rev^{2}ds\\
       & +
  4 C(m) ( T+4 )
       \E\int_{0}^{T_{1}\wedge\theta _{m}}\lev X_{\eta}(s)-\hat{X}(s) \rev^{2}ds .
\end{align*}
By the same reasoning which gave estimate \eqref{eq:numbound}, we can deduce that
\begin{align*}
\E\int_{0}^{T_{1}\wedge\theta _{m}}\lev X_{\eta}(s)-\hat{X}(s) \rev^{2}ds
\le
C(m,T)\D t.
\end{align*}
Hence
\begin{align*}
&\E
\left[
\sup_{0\leq t\leq T_{1}}\lev \hat{X}(t\wedge\theta_{m})-x(t\wedge \theta _{m})\rev ^{2}
\right]  \\
   &\leq  4 C(m) ( T+4 ) C(m,T)\D t + 4 C(m) ( T+4 )
 \E \left[
           \int_{0}^{T_{1}}\sup_{0\leq t\leq s}\lev \hat{X}(t\wedge \theta _{m})-x(t\wedge \theta_{m})\rev ^{2} ds
      \right].
\end{align*}
The statement of the theorem holds by the Gronwall inequality.
\end{proof}

\bibliographystyle{plain}
\bibliography{upthesis}

\begin{thebibliography}{10}

\bibitem{ahn1999parametric}
D.H. Ahn and B.~Gao.
\newblock {A parametric nonlinear model of term structure dynamics}.
\newblock {\em Review of Financial Studies}, 12(4):721, 1999.

\bibitem{ait1996testing}
Y.~Ait-Sahalia.
\newblock {Testing continuous-time models of the spot interest rate}.
\newblock {\em Review of Financial Studies}, 9(2):385--426, 1996.

\bibitem{bahar2004stochastic}
A.~Bahar and X.~Mao.
\newblock {Stochastic delay population dynamics}.
\newblock {\em International Journal of Pure and Applied Mathematics},
  11:377--400, 2004.

\bibitem{baker2005exponential}
C.T.H. Baker and E.~Buckwar.
\newblock {Exponential stability in p-th mean of solutions, and of convergent
  Euler-type solutions, of stochastic delay differential equations}.
\newblock {\em Journal of Computational and Applied Mathematics},
  184(2):404--427, 2005.

\bibitem{berkaoui2007euler}
A.~Berkaoui, M.~Bossy, and A.~Diop.
\newblock {Euler scheme for SDEs with non-Lipschitz diffusion coefficient:
  strong convergence}.
\newblock {\em ESAIM: Probability and Statistics}, 12:1--11, 2007.

\bibitem{broadie1997continuity}
M.~Broadie, P.~Glasserman, and S.~Kou.
\newblock {A continuity correction for discrete barrier options}.
\newblock {\em Mathematical Finance}, 7(4):325--349, 1997.

\bibitem{buchmann2005simulation}
F.M. Buchmann.
\newblock {Simulation of stopped diffusions}.
\newblock {\em Journal of Computational Physics}, 202(2):446--462, 2005.

\bibitem{campbell1998econometrics}
J.Y. Campbell, A.W. Lo, A.C. MacKinlay, and R.F. Whitelaw.
\newblock {The econometrics of financial markets}.
\newblock {\em Macroeconomic Dynamics}, 2(04):559--562, 1998.

\bibitem{chan1992empirical}
K.C Chan, G.A. Karolyi, F.A. Longstaff, and A.B. Sanders.
\newblock {An empirical comparison of alternative models of the short-term
  interest rate}.
\newblock {\em The journal of finance}, 47(3):1209--1227, 1992.

\bibitem{friedman1976}
A.~Friedman.
\newblock {\em {Stochastic differential equations and applications}}.
\newblock Academic Press, 1976.

\bibitem{gard1988introduction}
T.C. Gard.
\newblock {\em {Introduction to Stochastic Differential Equations}}.
\newblock Marcel Dekker, New York, 1988.

\bibitem{giles2008multilevel}
M.B. Giles.
\newblock {Multilevel monte carlo path simulation}.
\newblock {\em Operations Research-Baltimore}, 56(3):607--617, 2008.

\bibitem{hale1993introduction}
J.K. Hale and S.M.V. Lunel.
\newblock {\em {Introduction to Functional Differential Equations}}.
\newblock Springer Verlag, 1993.

\bibitem{heston1997simple}
S.L. Heston.
\newblock {A simple new formula for options with stochastic volatility}.
\newblock {\em Course notes of Washington University in St. Louis, Missouri},
  1997.

\bibitem{higham2000stability}
D.J. Higham.
\newblock {A-stability and stochastic mean-square stability}.
\newblock {\em BIT Numerical Mathematics}, 40(2):404--409, 2000.

\bibitem{higham2001mean}
D.J. Higham.
\newblock {Mean-square and asymptotic stability of the stochastic theta
  method}.
\newblock {\em SIAM Journal on Numerical Analysis}, pages 753--769, 2001.

\bibitem{higham2005convergence}
D.J. Higham and X.~Mao.
\newblock {Convergence of Monte Carlo simulations involving the mean-reverting
  square root process}.
\newblock {\em Journal of Computational Finance}, 8(3):35, 2005.

\bibitem{higham2003strong}
D.J. Higham, X.~Mao, and A.M. Stuart.
\newblock {Strong convergence of Euler-type methods for nonlinear stochastic
  differential equations}.
\newblock {\em SIAM Journal on Numerical Analysis}, 40(3):1041--1063, 2002.

\bibitem{higham2003exponential}
D.J. Higham, X.~Mao, and A.M. Stuart.
\newblock {Exponential mean-square stability of numerical solutions to
  stochastic differential equations}.
\newblock {\em LMS J. Comput. Math}, 6:297--313, 2003.

\bibitem{higham2008almost}
D.J. Higham, X.~Mao, and C.~Yuan.
\newblock {Almost sure and moment exponential stability in the numerical
  simulation of stochastic differential equations}.
\newblock {\em SIAM Journal on Numerical Analysis}, 45(2):592--609, 2008.

\bibitem{hu1996semi}
Y.~Hu.
\newblock {Semi-implicit Euler-Maruyama scheme for stiff stochastic equations}.
\newblock {\em Progress in Probability}, pages 183--202, 1996.

\bibitem{hutzenthaler2010strong}
M.~Hutzenthaler, A.~Jentzen, and P.E. Kloeden.
\newblock Strong convergence of an explicit numerical method for sdes with
  non-globally lipschitz continuous coefficients.
\newblock {\em to appear in The Annals of Applied Probability}, 2010.

\bibitem{hutzenthaler2011strong}
M.~Hutzenthaler, A.~Jentzen, and P.E. Kloeden.
\newblock {Strong and weak divergence in finite time of Euler's method for
  stochastic differential equations with non-globally Lipschitz continuous
  coefficients}.
\newblock {\em Proceedings of the Royal Society A: Mathematical, Physical and
  Engineering Science}, 467(2130):1563, 2011.

\bibitem{chas1980stochastic}
R.Z. Khasminski.
\newblock {\em {Stochastic Stability of Differential Equations}}.
\newblock Kluwer Academic Pub, 1980.

\bibitem{kloeden2007pathwise}
P.E. Kloeden and A.~Neuenkirch.
\newblock {The pathwise convergence of approximation schemes for stochastic
  differential equations}.
\newblock {\em Journal of Computation and Mathematics}, 10:235--253, 2007.

\bibitem{kloeden1992numerical}
P.E. Kloeden and E.~Platen.
\newblock {\em {Numerical Solution of Stochastic Differential Equations}}.
\newblock Springer, 1992.

\bibitem{lasalle1968stability}
J.P. LaSalle.
\newblock {Stability theory for ordinary differential equations}.
\newblock {\em J. Differential Equations}, 4(1):57--65, 1968.

\bibitem{lasalle1976stability}
J.P. LaSalle and Z.~Artstein.
\newblock {\em {The Stability of Dynamical Systems}}.
\newblock Society for Industrial Mathematics, 1976.

\bibitem{lewis2000option}
A.L. Lewis.
\newblock {\em {Option Valuation Under Stochastic Volatility}}.
\newblock Finance Press, 2000.

\bibitem{liptser1989theory}
R.S. Liptser and A.N. Shiryayev.
\newblock {\em {Theory of Martingales}}.
\newblock Kluwer Academic Publishers, 1989.

\bibitem{lucia2002electricity}
J.J. Lucia and E.S. Schwartz.
\newblock Electricity prices and power derivatives: Evidence from the nordic
  power exchange.
\newblock {\em Review of Derivatives Research}, 5(1):5--50, 2002.

\bibitem{mannella1999absorbing}
R.~Mannella.
\newblock {Absorbing boundaries and optimal stopping in a stochastic
  differential equation}.
\newblock {\em Physics Letters A}, 254(5):257--262, 1999.

\bibitem{mao1999stochastic}
X.~Mao.
\newblock {Stochastic versions of the LaSalle theorem}.
\newblock {\em Journal of Differential Equations}, 153(1):175--195, 1999.

\bibitem{mao2007stochastic}
X.~Mao.
\newblock {\em {Stochastic Differential Equations and Applications}}.
\newblock Horwood Pub Ltd, 2007.

\bibitem{mao2002environmental}
X.~Mao, G.~Marion, and E.~Renshaw.
\newblock {Environmental Brownian noise suppresses explosions in population
  dynamics}.
\newblock {\em Stochastic Process. Appl}, 97(1):95--110, 2002.

\bibitem{mao2003asymptotic}
X.~Mao, S.~Sabanis, and E.~Renshaw.
\newblock {Asymptotic behaviour of the stochastic Lotka--Volterra model}.
\newblock {\em Journal of Mathematical Analysis and Applications},
  287(1):141--156, 2003.

\bibitem{szpruch-diss}
X.~Mao and L.~Szpruch.
\newblock {Strong convergence rates for backward Euler–Maruyama method for
  non-linear dissipative-type stochastic differential equations with
  super-linear diffusion coefficients}.
\newblock {\em to appear in Stochastics}, 2012.

\bibitem{mao2006stochastic}
X.~Mao and C.~Yuan.
\newblock {\em {Stochastic Differential Equations with Markovian Switching}}.
\newblock Imperial College Press, 2006.

\bibitem{mao2007approximations}
X.~Mao, C.~Yuan, and G.~Yin.
\newblock {Approximations of Euler--Maruyama type for stochastic differential
  equations with Markovian switching, under non-Lipschitz conditions}.
\newblock {\em Journal of Computational and Applied Mathematics},
  205(2):936--948, 2007.

\bibitem{mattingly2002ergodicity}
J.C. Mattingly, A.M. Stuart, and DJ~Higham.
\newblock Ergodicity for sdes and approximations: locally lipschitz vector
  fields and degenerate noise.
\newblock {\em Stochastic processes and their applications}, 101(2):185--232,
  2002.

\bibitem{pang-asymptotic}
S.~Pang, F.~Deng, and X.~Mao.
\newblock {Asymptotic properties of stochastic population dynamics}.
\newblock {\em Dynamics of Continuous, Discrete and Impulsive Systems Series
  A}, 15:603--620, 2008.

\bibitem{shen2006improved}
Y.~Shen, Q.~Luo, and X.~Mao.
\newblock {The improved LaSalle-type theorems for stochastic functional
  differential equations}.
\newblock {\em Journal of Mathematical Analysis and Applications},
  318(1):134--154, 2006.

\bibitem{szpruchnumerical}
L.~Szpruch, X.~Mao, D.J. Higham, and J.~Pan.
\newblock {Numerical simulation of a strongly nonlinear Ait-Sahalia-type
  interest rate model}.
\newblock {\em BIT Numerical Mathematics}, pages 1--21.

\end{thebibliography}

\end{document}